\documentclass{amsart}

\usepackage{amssymb}

\usepackage[urlcolor=black, colorlinks=true, citecolor=black, linkcolor=black]{hyperref}
\usepackage{tikz}
\tikzset{node distance=3cm, auto}
\usetikzlibrary{calc}
\usetikzlibrary{patterns}
\usepackage{tabulary}
\usepackage{array}
\usepackage{enumerate,enumitem}
\usepackage{bbm}

\makeatletter
\@namedef{subjclassname@2010}{%
  \textup{2010} Mathematics Subject Classification}
\makeatother

\newtheorem*{theorem*}{Theorem}
\newtheorem{mythm}{Theorem} \numberwithin{mythm}{section} 

\newtheorem{myobs}[mythm]{Observation}
\newtheorem{mycor}[mythm]{Corollary}
\newtheorem{mylem}[mythm]{Lemma} 
\newtheorem{myquest}{Question} 

\newtheorem*{myclaim*}{Claim}

\theoremstyle{definition}

\numberwithin{equation}{section}

\newcommand{ \N } { \mathbb{N} }

\newcommand{\w}{\omega}
\newcommand{\cont}{\mathfrak{c}}
\newcommand{\continuum}{\mathfrak{c}}

\newcommand{\script}{\mathcal}
\newcommand{\parentheses}[1]{{\left( {#1} \right)}}
\newcommand{\sequence}[1]{{\langle {#1} \rangle}}
\newcommand{\p}{\parentheses}

\newcommand{\Set}[1]{{\left\lbrace {#1} \right\rbrace}}
\newcommand{\singleton}{\Set}
\newcommand{\cardinality}[1]{{\left\lvert {#1} \right\rvert}}

\newcommand{\pair}[1]{\langle {#1} \rangle}
\def\set#1:#2{\Set{{#1} \colon {#2}}}
\def\Sequence#1:#2{\left \langle {#1} \colon {#2}  \right \rangle}

\renewcommand{\subset}{\subseteq}

\begin{document}


\baselineskip=17pt



\title[Minimal obstructions for normal spanning trees]{Minimal obstructions for normal spanning trees}

\author[N. Bowler]{Nathan Bowler}
\address{Department of Mathematics\\ University of Hamburg\\ Bundesstra{\ss}e 55, 20146 Hamburg, Germany}
\email{nathan.bowler@uni-hamburg.de}

\author[S. Geschke]{Stefan Geschke}
\address{Department of Mathematics\\ University of Hamburg\\ Bundesstra{\ss}e 55, 20146 Hamburg, Germany}
\email{stefan.geschke@uni-hamburg.de}

\author[M. Pitz]{Max Pitz}
\address{Department of Mathematics\\ University of Hamburg\\ Bundesstra{\ss}e 55, 20146 Hamburg, Germany}
\email{max.pitz@uni-hamburg.de}
\thanks{The third author is the corresponding author.}

\date{}

\begin{abstract}
Diestel and Leader have characterised connected graphs that admit a normal spanning tree via two classes of forbidden minors. One class are Halin's $(\aleph_0,\aleph_1)$-graphs: bipartite graphs with bipartition $(A,B)$ such that $\cardinality{A} = \aleph_0$, $\cardinality{B} = \aleph_1$ and every vertex of $B$ has infinite degree. 

Our main result is that under Martin's Axiom and the failure of the Continuum Hypothesis, the class of forbidden $(\aleph_0,\aleph_1)$-graphs in Diestel and Leader's result can be replaced by one single instance of such a graph.

Under CH, however, the class of $(\aleph_0,\aleph_1)$-graphs contains minor-incom-parable elements, namely graphs of binary type, and $\script{U}$-indivisible graphs. Assuming CH, Diestel and Leader asked  whether every $(\aleph_0,\aleph_1)$-graph has an $(\aleph_0,\aleph_1)$-minor that is either indivisible or of binary type, and whether any two $\script{U}$-indivisible graphs are necessarily minors of each other. For both questions, we construct examples showing that the answer is in the negative.
\end{abstract}


\maketitle

\section{The results}

A \emph{(graph theoretic) tree} is a connected, acyclic graph. A subgraph $H$ of a graph $G$ is called \emph{spanning} if $H$ has the same vertex set as $G$. Thus, a \emph{spanning tree} $T$ of a connected graph $G$ is a connected, acyclic subgraph containing every vertex of $G$. A tree is \emph{rooted} if it has one designated vertex, called the \emph{root}. Fixing a root of a graph-theoretic tree $T$ induces a natural tree order on its vertex set $V(T)$ with the root as unique minimal element.

A rooted spanning tree $T$ of a graph $G$ is called \emph{normal} if the end-vertices of any edge of $G$ are comparable in the natural tree order of $T$, see e.g\ \cite[\S1.5]{Diestel}. Intuitively, all the edges of $G$ run `parallel' to branches of $T$, but never `across'. Every countable connected graph has a normal spanning tree, but uncountable graphs might not, as demonstrated by complete graphs on uncountably many vertices \cite[8.2.3]{Diestel}.

Halin \cite[7.2]{Halin} observed that as a consequence of a theorem of Jung, the property of having a normal spanning tree is minor-closed, i.e.\ preserved under taking (connected) minors. Here, a graph $H$ is a \emph{minor} of another graph $G$, written $H \preceq G$, if to every vertex $x \in H$ we can assign a (possibly infinite) connected set $V_x \subset V(G)$, called the \emph{branch set} of $x$, so that these sets $V_x$ are disjoint for different $x$ and $G$ contains a $V_x-V_y$ edge whenever $xy$ is an edge of $H$. 

Halin's observation opens up the possibility of a forbidden minor characterisation for the property of admitting normal spanning trees. In the universe of finite graphs, the famous Seymour-Robertson Theorem asserts that any minor-closed property of finite graphs can be characterised by \emph{finitely} many forbidden minors, see e.g.\ \cite[\S12.7]{Diestel}. Whilst for infinite graphs, we generally need an infinite list of forbidden minors, Diestel and Leader have shown that for the property of having a normal spanning tree, the forbidden minors come in two structural types. 

Following Halin, a bipartite graph with bipartition $(A,B)$ is called an $(\aleph_0,\aleph_1)$\emph{-graph} if $\cardinality{A}=\aleph_0$, $\cardinality{B}=\aleph_1$, and every vertex in $B$ has infinite degree. 


\begin{theorem*}[Diestel and Leader, \cite{NST}]
A connected graph admits a normal spanning tree if and only if it does not contain an $(\aleph_0,\aleph_1)$-graph or an AT-graph (a certain kind of graph whose vertex set is an order-theoretic Aronszajn tree) as a minor.
\end{theorem*}

In the same paper, they ask how one might further describe the minor-minimal graphs within the class of $(\aleph_0,\aleph_1)$-graphs. 

One family of possibly minimal $(\aleph_0,\aleph_1)$-graphs suggested by Diestel and Leader are the \emph{binary trees with tops}, also called $(\aleph_0,\aleph_1)$-graphs of \emph{binary type}: Let $A$ be a binary tree of countable height, and let $B$ index $\aleph_1$-many branches of $A$. We form an $(\aleph_0,\aleph_1)$-graph with bipartition $(A,B)$ by connecting every vertex $b \in B$ to infinitely many points on its branch. Details on these graphs can be found in Section~\ref{sec2}. We can now state our main result as follows.

\begin{mythm}
\label{intromain}
Let $T$ be an arbitrary binary tree with tops. Under Martin's Axiom and the failure of the Continuum Hypothesis, the graph $T$ embeds into any other $(\aleph_0,\aleph_1)$-graph as a subgraph.
\end{mythm}

Answering a question by Diestel and Leader, it follows that it is consistent with the usual axioms of set theory ZFC that there is a minor-minimal graph without a normal spanning tree. As a second consequence, we can extend Diestel and Leader's result as follows.
\begin{mythm}
\label{supermain}
Let $T$ be an arbitrary binary tree with tops. Under Martin's Axiom and the failure of the Continuum Hypothesis, a graph has a normal spanning tree if and only if it does not contain $T$, or an AT-graph as a minor. 
\end{mythm}

However, under the Continuum Hypothesis (CH) the situation is different. Now, there exist \emph{indivisible} $(\aleph_0,\aleph_1)$-graphs, i.e.\ graphs $(\N,B)$ where for every partition $\N=A_1\dot\cup A_2$, only one of the induced graphs $(A_1,B)$ and $(A_2,B)$ contains an $(\aleph_0,\aleph_1)$-subgraph.
Note that for every indivisible graph $(\N,B)$ there is a corresponding (non-principal) ultrafilter $\script{U}$ consisting of all subsets $A \subset \N$ such that $(A,B)$ contains an $(\aleph_0,\aleph_1)$-subgraph. Indivisible graphs with associated ultrafilter $\script{U}$ are also called $\script{U}$\emph{-indivisible}. 

In \cite[8.1]{NST}, Diestel and Leader proved that binary trees with tops and indivisible graphs form two minor-incomparable classes of $(\aleph_0,\aleph_1)$-graphs. Further, they mention the following two problems involving indivisible graphs:

\begin{myquest}[Diestel and Leader]
\label{athirdclass}
Assuming CH, does every $(\aleph_0,\aleph_1)$-graph have an $(\aleph_0,\aleph_1)$-minor that is either indivisible or of binary type?
\end{myquest}

\begin{myquest}[Diestel and Leader]
\label{Uindivisible}
Assuming CH, are any two $\script{U}$-indivisible $(\aleph_0,\aleph_1)$-graphs necessarily minors of each other?
\end{myquest}

One particular property of $(\aleph_0,\aleph_1)$-graphs of binary type is that they are \emph{almost disjoint} (AD): neighbourhoods of any two distinct $B$-vertices intersect only finitely (see Section~\ref{sec2} for further details). Of course, not every $(\aleph_0,\aleph_1)$-graph has this property, as complete bipartite graphs show. However, our first result in this paper is that we can always restrict our attention to almost disjoint $(\aleph_0,\aleph_1)$-graphs: In Theorem~\ref{ADsubgraph} below, we show that every $(\aleph_0,\aleph_1)$-graph has an AD-$(\aleph_0,\aleph_1)$-subgraph.

Once we have made this reduction, we turn towards Questions~\ref{athirdclass} and \ref{Uindivisible}. In Theorem~\ref{thm_cool}, we show that Question~\ref{athirdclass} has a negative answer. Our construction refines a strategy developed by Roitman and Soukup for the combinatorical analysis of almost disjoint families. We then construct in Theorem~\ref{nathansthm} two $\script{U}$-indivisible graphs that are not minor-equivalent, answering Question~\ref{Uindivisible} in the negative.

\section{\texorpdfstring{Collections of infinite subsets of $\N$, and $(\aleph_0,\aleph_1)$-graphs}%
                        {Collections of infinite subsets of N, and A0A1-graphs}}
\label{sec2}

The following connection between collections of infinite subsets of $\N$ and $(\aleph_0,\aleph_1)$-graphs will be used frequently in this paper. Let $G$ be an $(\aleph_0,\aleph_1)$-graph with bipartition $(A,B)$, and enumeration $B=\set{b_\alpha}:{\alpha < \w_1}$. Identifying $A$ with the integers $\N$, we can encode $G$ as (multi-)set $\sequence{N(b_\alpha) \colon \alpha < \w_1}$ of infinite subsets of $\N$. Conversely, given any multiset $\sequence{N_\alpha \colon \alpha < \w_1}$ of infinite subsets of $\N$, we can form an $(\aleph_0,\aleph_1)$-graph with bipartition $(\N,B)$ by setting $N(b_\alpha):=N_\alpha$.

This correspondence allows us to translate graph-theoretic problems about $(\aleph_0,\aleph_1)$-graphs to the realm of infinite combinatorics. Let $A$ and $B$ be subsets of $\N$. If $A \setminus B$ is finite, we say that $A$ is \emph{almost contained} in $B$, or $A$ is contained in $B$ \emph{mod finite}, and write $A \subset^* B$.   Consequently, $A$ and $B$ are \emph{almost equal}, $A =^* B$, if $A \subset ^* B$ and $B \subset^* A$ (which means their symmetric difference is finite). 

Given any collection $\script{P}$ of infinite subsets of $\N$, we say that an infinite set $A \subset \N$ is a \emph{pseudo-intersection} for $\script{P}$ if $A \subset^* P$ for all $P \in \script{P}$. Every countable $\script{P}$ that is directed by $\subset^*$ has a pseudo-intersection.

A collection $\mathcal{A}$ of infinite subsets of $\N$ is an \emph{almost disjoint family} (AD-family) if $A \cap A' =^* \emptyset$ for all $A,A'$ in $\mathcal{A}$ (in other words, if the pairwise intersection of elements of $\script{A}$ is always finite). By a diagonalisation argument, every infinite AD-family can be extended to an uncountable AD-family. 

The simplest example of an $(\aleph_0,\aleph_1)$-graph is the complete bipartite graph $K_{\aleph_0,\aleph_1}$. Binary trees with tops as introduced above are strictly smaller (with respect to the minor relation $\preceq$) examples of $(\aleph_0,\aleph_1)$-graphs, as they have the property that $\cardinality{N(b) \cap N(b')} < \infty$ for all $b\neq b' \in B$. Changing our perspective, we see that in this case, the collection $\sequence{N(b_\alpha) \colon \alpha < \w_1}$ forms an almost disjoint family on $\N$. Let us call any $(\aleph_0,\aleph_1)$-graph with this last property an \emph{almost disjoint} $(\aleph_0,\aleph_1)$\emph{-graph}, or for short an AD-$(\aleph_0,\aleph_1)$-graph.

A \emph{tree} $\script{T} = \p{T, <}$ in the order-theoretic sense is a partially ordered set $T$ with a smallest element such that all predecessor sets $t^{\downarrow} = \set{s \in T}:{s < t}$ are well-ordered by $<$. The order type of $t^{\downarrow}$ is called the \emph{height} of $t$, and denoted by $\operatorname{ht}(t)$. The set of all elements of $\script{T}$ of height $\alpha$ is denoted by $\script{T}(\alpha)$, and called the $\alpha^{\text{th}}$ level of $\script{T}$. A subset $S \subset T$ of a tree $\script{T}=\p{T,<}$ is an \emph{initial subtree} if $t^{\downarrow} \subset S$ for all $t \in S$. By $\script{T}(\leq \alpha) = \bigcup_{\beta \leq \alpha} T(\beta)$ we mean the initial subtree of $\script{T}$ consisting of all elements of $\script{T}$ of height at most $\alpha$.

A linearly ordered subset of $\script{T}$ is also called a \emph{chain}. A \emph{branch} of a tree $\script{T}$ is an inclusion-maximal chain. The collection of branches is also denoted by $\script{B}(\script{T})$. For $b$ a branch and $\alpha$ an ordinal, $b \restriction \alpha$ denotes the unique element of $b \cap T(\alpha)$. An \emph{Aronszajn tree} is an uncountable tree such that all levels and all branches are countable. The \emph{binary tree of countable height} is the tree $2^{<\omega}$, the set of all finite binary sequences, ordered by extension. Similarly, a \emph{binary tree of finite height} is a tree isomorphic to $2^{<\w}({\leq}n)$ for some $n \in \N$.

In the following, we list some special types of $(\aleph_0,\aleph_1)$-graphs (suggested by Diestel and Leader \cite{NST}), and some well-known types of almost disjoint families (studied by Roitman and Soukup \cite{Luzin}), all of which will play a role in this paper.

\subsection*{ Graph-theoretic perspective (Diestel \& Leader)}
\begin{itemize}
\item $T_2^{tops}$: Let $A=2^{<\w}$ be a binary tree of height $\w$, and $B$ be a set of $\aleph_1$-many branches of $A$. Any graph isomorphic to some $(\aleph_0,\aleph_1)$-graph formed on the vertex set $A \dot\cup B$ by connecting every vertex $b \in B$ to infinitely many points on its branch is called a $T_2^{tops}$, or an $(\aleph_0,\aleph_1)$-graph of binary type.
\item \emph{full $T_2^{tops}$}: As above, but now connect every vertex $b \in B$ to all points on its branch. 
\item \emph{divisible:} An $(\aleph_0,\aleph_1)$-graph with bipartition  $(A,B)$ is divisible if there are partitions $A=A_1 \dot\cup A_2$ and $B=B_1 \dot\cup B_2$ such that both $(A_1,B_1)$ and $(A_2,B_2)$ contain $(\aleph_0,\aleph_1)$-subgraphs.
\item \emph{$\script{U}$-indivisible:} For a non-principal ultrafilter $\script{U}$, an $(\aleph_0,\aleph_1)$-graph with bipartition $(\N,B)$ is called $\script{U}$-indivisible if for all $A \in \script{U}$ we have $N(b) \subset^* A$ for all but countably many $b \in B$.
\end{itemize}

\subsection*{Set-theoretic perspective (Roitman \& Soukup)} 
\begin{itemize}
\item \emph{tree-family:} An uncountable AD-family $\script{A}$ on $\N$ is a tree-family if there is a tree-ordering $\script{T}$ of countable height on $\N$ so that for every $A \in \script{A}$ there is a branch of $\script{T}$ which almost equals $A$.
\item \emph{weak tree-family:} As above, but now it is only required that there is an injective assignment from $\script{A}$ to branches of $\script{T}$ such that every $A \in \script{A}$ is almost contained in its assigned branch.
\item \emph{hidden (weak) tree-family:} $\script{A}$ is a hidden (weak) tree family if for some countable tree $T$, $\set{T\cap a}:{a \in \script{A}}$ a (weak) tree family.
\item \emph{anti-Luzin:} An AD-family $\script{A}$ is anti-Luzin if for all uncountable $\script{B}\subset \script{A}$ there are uncountable $\script{C},\script{D} \subset \script{B}$ such that $\bigcup \script{C} \cap \bigcup \script{D}$ is finite.
\end{itemize}

\subsection*{Comparing the different notions} 
There are striking similarities between the graph-theoretic and the set-theoretic perspective. We gather dependencies between the above concepts in the following diagram. All these implications are straightforward from the definitions.

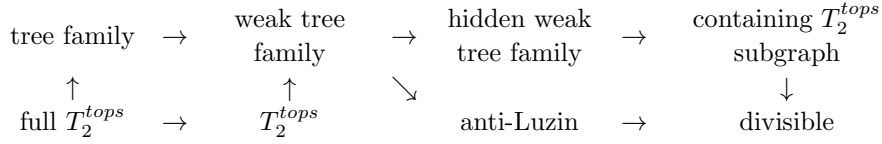
\newcolumntype{C}[1]{>{\centering\let\newline\\\arraybackslash\hspace{0pt}}m{#1}}
\begin{center}
\begin{tabular}{cc C{2cm} c C{2cm} c C{3cm}}
tree family  & $\rightarrow$ & weak tree family & $\rightarrow$ & hidden weak tree family & $\rightarrow$ & containing $T_2^{tops}$ subgraph \\
$\uparrow$ &    & $\uparrow$ & $\searrow$ &  &   & $\downarrow$ \\ 
full $T_2^{tops}$& $\rightarrow$ & $T_2^{tops}$ &  &  anti-Luzin & $\rightarrow$ & divisible \\
\end{tabular}
\end{center}

A little less straightforward is the fact that none of the arrows in the above diagram can generally be reversed. This is witnessed by the following examples.

\begin{myobs}
Under CH, there is a binary tree with tops which is not a tree family.
\end{myobs}
\begin{proof}[Construction sketch] Consider a binary tree order $\script{T}$ on $\N$ and, using CH, enumerate its branches $\script{B}(\script{T}) = \set{b_\alpha}:{\alpha < \w_1}$. In order to diagonalize against all possible tree families, enumerate all tree orders of countable height on $\N$ as $\set{\script{T}_\alpha}:{\alpha < \w_1}$. Now if $\cardinality{b_\alpha \cap b} = \infty$ for some branch $b$ of $\script{T}_\alpha$, then choose $N_\alpha \subset b_\alpha$ such that $N_\alpha \subsetneq^* b$. Otherwise, put $N_\alpha = b_\alpha$. Then $\Sequence{N_\alpha}:{\alpha < \omega_1}$ is as desired. 
\end{proof}

Hence, the implications `tree family $\rightarrow$ weak tree family' and `full $T_2^{tops} \rightarrow T_2^{tops}$' cannot be reversed.

Next, if in a full $T_2^{tops}$ one additionally makes all tops adjacent to one special node of the tree, one obtains a tree family which cannot be a $T_2^{tops}$, because in a $T_2^{tops}$ without isolated points on the countable side, only the root of the tree can be simultaneously adjacent to all tops. In particular, the implications `$T_2^{tops} \rightarrow$ weak tree family' and `full $T_2^{tops} \rightarrow$ tree family' cannot be reversed.

Hidden weak tree families need not be anti-Luzin, see \cite[p.58]{Luzin}. In particular, the implications `weak tree family $\rightarrow$ hidden weak tree family' and `anti-Luzin $\rightarrow$ divisible' cannot be reversed. In Theorem~\ref{thm_cool} below, we construct under CH an anti-Luzin family which contains no $T_2^{tops}$ subgraph, so the implications `weak tree family $\rightarrow$ anti-Luzin' and `containing $T_2^{tops}$ subgraph $\rightarrow$ divisible' cannot be reversed. Finally, the implication `hidden weak tree family $\rightarrow$ containing a $T_2^{tops}$ subgraph' cannot be reversed:

\begin{myobs}
Under CH, there is an AD-family $\Sequence{N_\alpha}:{\alpha < \omega_1}$ containing a $T_2^{tops}$ subgraph but which is not a hidden weak tree family.
\end{myobs}
\begin{proof}[Construction Sketch]
Consider a binary tree order $\script{T}$ on $\N$ and enumerate its branches $\script{B}(\script{T}) = \set{b_\alpha}:{\alpha < \w_1}$. Enumerate all tree orders of countable height with groundset some infinite subset of $\N$ as $\set{\script{T}_\alpha}:{\alpha < \w_1}$. Every $N_\alpha$ will be the union of at most two $b_{\beta_1(\alpha)}$ and $b_{\beta_2(\alpha)}$. At step $\alpha < \w_1$, we have $\beta=\sup \set{b_{\beta_1(\gamma)},b_{\beta_2(\gamma)}}:{\gamma<\alpha} < \w_1$. If there is $b_\delta$ with $\delta > \beta$ such that $b_\delta$ is not almost contained in a single branch of $\script{T}_\alpha$, put $N_\alpha = b_\delta$. If all $b_\delta$ with $\delta > \beta$ are almost contained in the same branch of $\script{T}_\alpha$, put $N_\alpha = b_{\beta+1}$. Otherwise, there are $\beta_1(\alpha) > \beta$ and $\beta_2(\alpha) > \beta$ such that $b_{\beta_1(\alpha)}$ and $b_{\beta_2(\alpha)}$ are almost contained in different branches of $\script{T}_\alpha$. Put $N_\alpha = b_{\beta_1(\alpha)} \cup b_{\beta_2(\alpha)}$.  Then it is easily checked that $\Sequence{N_\alpha}:{\alpha < \omega_1}$ is as desired. 
\end{proof}

However, under MA+$\neg$CH, every $<\cont$-sized AD family is a hidden weak tree family \cite[4.4]{Luzin}, so the last construction cannot be done in ZFC alone.

\section{\texorpdfstring{Finding almost disjoint $(\aleph_0,\aleph_1)$-subgraphs}%
                        {Finding almost disjoint A0A1-subgraphs}}

Almost disjoint $(\aleph_0,\aleph_1)$-graphs are natural candidates for smaller obstruction sets in Diestel and Leader's result. In this section, we prove that indeed, every $(\aleph_0,\aleph_1)$-graph contains an almost disjoint $(\aleph_0,\aleph_1)$-subgraph. 

We say that a collection $\mathcal{F}$ of infinite subsets of some countably infinite set has an \emph{almost disjoint refinement} if there is a choice of infinite subsets $A_F \subset F$ such that $\mathcal{A}=\set{A_F}:{F \in \mathcal{F}}$ is an almost disjoint family.

\begin{mythm}[Baumgartner, Hajnal and Mate; Hechler]
\label{refinement}
Every $< \continuum$-sized collection of infinite subsets of $\N$ has an almost disjoint refinement.
\end{mythm}

The theorem is due to Baumgartner, Hajnal and Mate \cite[2.1]{Baum}, and independently due to Hechler \cite[2.1]{Hechler}. For convenience, we will indicate the proof below.

\begin{mycor}
\label{cor_refinement}
Assume $\neg CH$. Every $(\aleph_0,\aleph_1)$-graph has a spanning AD-$(\aleph_0,\aleph_1)$-subgraph. 
\end{mycor}
\begin{proof}
An almost disjoint refinement corresponds, in the graph-theoretic perspective, to a subgraph obtained by deleting, at every vertex on the $B$-side, co-infinitely many incident edges. Since we did not remove any vertices, we obtain indeed a spanning AD-$(\aleph_0,\aleph_1)$-subgraph. 
\end{proof}

Theorem \ref{refinement} does not hold for families of size $\continuum$ (consider the collection of all infinite subsets of $\N$). Still, we can prove that the corresponding result for subgraphs is true nonetheless (but we can no longer guarantee spanning subgraphs).

\begin{mythm}
\label{ADsubgraph}
Every $(\aleph_0,\aleph_1)$-graph has an AD-$(\aleph_0,\aleph_1)$-subgraph.
\end{mythm}

First, a piece of notation. Let $\mathcal{F}$ be a collection of infinite subsets of $\N$, and $\script{A}$ be an almost disjoint family. Following Hechler, \cite{Hechler}, we say that $\script{A}$ \emph{covers} $\script{F}$ if for every $F \in \script{F}$, the collection $\set{A \in \script{A}}:{\cardinality{F \cap A}=\infty}$ is of size $\cardinality{\script{A}}$. 

Hechler showed that a collection $\script{F}$ of infinite subsets of $\N$ has an almost disjoint refinement if and only if there is an almost disjoint family of size $\cardinality{\script{F}}$ covering $\script{F}$ \cite[2.3]{Hechler}. We shall only make use of the backwards implication, the proof of which is nicely illustrated in the claim below.

\begin{proof}[Proof of Theorem~\ref{ADsubgraph}]
Suppose we are given an $(\aleph_0,\aleph_1)$-graph $G$ with bipartition $(\N,B)$, an enumeration $B=\set{b_\alpha}:{\alpha < \w_1}$ and neighbourhoods $N_\alpha = N(b_\alpha)$. 

\begin{myclaim*}
If $\set{N_\alpha}:{\alpha < \w_1}$ forms an uncountable decreasing chain mod finite (i.e.\ $N_\beta \subseteq^* N_\alpha$ for all $\alpha < \beta$), then $G$ has an AD-$(\aleph_0,\aleph_1)$-subgraph.
\end{myclaim*}

For the claim, consider two alternatives. Either, $\script{N}=\set{N_\alpha}:{\alpha < \w_1}$ has an infinite pseudo-intersection $A$, in which case any uncountable AD-family $\script{A}=\set{A_\alpha}:{\alpha < \w_1}$ on $A$ covers $\set{N_\alpha}:{\alpha < \w_1}$. Picking $N'_\alpha = N_\alpha \cap A_\alpha$ readily provides an almost disjoint refinement of $\script{N}$. And if $\script{N}$ does not have an infinite pseudo-intersection, then moving to a subgraph, we may assume that $C_\alpha= N_\alpha \setminus N_{\alpha + 1}$ is infinite for all $\alpha < \w_1$. Now if $\alpha < \beta$ then $C_\alpha \cap C_\beta \subset N_\alpha \setminus N_{\alpha +1 } \cap N_\beta$ is finite, as $N_\beta \setminus N_{\alpha +1}$ is finite by assumption. So $\set{C_\alpha}:{\alpha < \w_1}$ gives rise to an AD-$(\aleph_0,\aleph_1)$-subgraph of $G$, establishing the claim.

Now suppose there exists an infinite set $A \subset \N$ with the property that for every infinite $C \subset A$ there is an uncountable set $K_C=\set{\beta < \w_1}:{\cardinality{N_\beta \cap C}=\infty}$. Let us construct, by recursion,  
\begin{enumerate}
\item a faithfully indexed set $\set{N_{\mu_\alpha}}:{\alpha < \w_1} \subset \script{N}$, and
\item infinite subsets $C_\alpha \subset N_{\mu_\alpha} \cap A$ such that $C_\alpha \subset^* C_\beta$ for all $ \alpha > \beta$.
\end{enumerate} 

First, let $\mu_0=\min K_A$ and put $C_0 = A \cap N_{\mu_0}$, an infinite subset of $A$. Next, let $\alpha < \w_1$ and suppose $\mu_\beta$ and $C_\beta$ have been defined according to $(1)$ and $(2)$ for all $\beta < \alpha$. Let $\tilde{C}_\alpha$ be an infinite pseudo-intersection of the countable collection $\set{C_\beta}:{\beta < \alpha}$. We may assume that $\tilde{C}_\alpha \subset A$ and let $\mu_\alpha = \min \p{K_{\tilde{C}_\alpha} \setminus \set{\mu_\beta}:{\beta < \alpha} }$. Then $C_\alpha = \tilde{C}_\alpha \cap N_{\mu_\alpha}$ is as required.

Once the recursion is completed, we can move to the subgraph on $(A,\set{\mu_\alpha}:{\alpha < \w_1})$ with neighbourhoods $N(\mu_\alpha)$ given by $C_\alpha$. By property $(2)$, the claim applies and we obtain an AD-$(\aleph_0,\aleph_1)$-subgraph.

Thus, we can assume that every infinite subset of $\N$, and in particular every $N_\alpha$ contains an infinite subset $C_\alpha$ such that $K_{C_\alpha}$ is countable. Recursively, pick an increasing transfinite subsequence $\set{\nu_\alpha}:{\alpha<\w_1}$ of $\w_1$, defined recursively by $\nu_0=0$ and 

$$\nu_\alpha = \sup \p{\set{\nu_\beta}:{\beta < \alpha} \cup \bigcup_{\beta < \alpha} K_{C_{\nu_\beta}} }+ 1 < \w_1.$$ 
We claim that $\set{C_{\nu_\alpha}}:{\alpha < \w_1}$ gives rise to an AD-$(\aleph_0,\aleph_1)$-subgraph of $G$. It is a subgraph, since by construction, we have $C_{\nu_\alpha} \subset N(\nu_\alpha)$. And it is almost disjoint, since given two arbitrary neighbourhoods $C_{\nu_\alpha}$ and $C_{\nu_\beta}$ with say $\nu_\alpha < \nu_\beta$, we have $C_{\nu_\alpha} \cap C_{\nu_\beta} \subset C_{\nu_\alpha} \cap N_{\nu_\beta},$ which is finite since $\nu_\beta \notin K_{\nu_\alpha}$ by construction.
\end{proof}

For completeness, we provide the proof of Theorem~\ref{refinement}.

\begin{proof}[Proof of Theorem~\ref{refinement}]
Let $\mathcal{F}=\set{F_\alpha}:{\alpha < \kappa}$ be a $\kappa< \cont$ sized family of infinite subsets of $\N$. We want to find an almost disjoint family $\mathcal{B}=\set{B_\alpha}:{\alpha < \kappa}$ such that $B_\alpha \subset F_\alpha$ for all $\alpha < \kappa$. 

Step 1: Split each $F_\alpha$ into an almost disjoint family $\mathcal{S}_\alpha=\set{S^\alpha_\xi}:{\xi < \kappa^+}$, i.e.\ all $S^\alpha_\xi$ are infinite subsets of $F_\alpha$, and $S^\alpha_\xi \cap S^\alpha_\zeta$ is finite whenever $\xi \neq \zeta < \kappa^+$. As $\kappa^+ \leq \cont$, this is always possible. Note that $\kappa^+$ is a regular cardinal.

Step 2: From our definition of `covering' after Theorem~\ref{ADsubgraph}, it follows that a $\kappa^+$-sized AD-family $S_\alpha$ covers $\singleton{F_\beta}$ iff $\set{S^\alpha_\xi \cap F_\beta}:{\cardinality{S^\alpha_\xi \cap F_\beta}=\infty}$ is a $\kappa^+$-sized AD-family on $F_\beta$. For all $\alpha < \kappa$ we use
\[Y_\alpha =\set{\beta < \kappa}:{\mathcal{S}_\alpha \textnormal{ covers } \Set{F_\beta}}\] to build a partition of $\kappa$ into (possibly empty) sets $\set{X_\alpha}:{\alpha<\kappa}$, defined by $X_0 = Y_0$ and $X_\alpha = Y_\alpha \setminus \bigcup_{\beta < \alpha} Y_\beta$.

Step 3: 
For all $\alpha \notin Y_\beta$ there is $\kappa(\alpha,\beta) < \kappa^+$ such that $\cardinality{F_\alpha \cap S^\beta_\xi} < \infty$ for all $\xi \geq \kappa(\alpha,\beta)$.
Define
\[\eta = \sup \set{\kappa(\alpha,\beta)}:{\beta < \kappa, \alpha \notin Y_\beta} < \kappa^+.\]

Step 4: Here, we pick the almost disjoint refinement. For all $\beta$ there is $\alpha(\beta)$ such that $\beta \in X_{\alpha(\beta)}$. For all $\beta \in X_\alpha$ we choose different $\xi(\beta) > \eta$ and define $B_\beta = S^{\alpha(\beta)}_{\xi(\beta)} \cap F_\beta$. Since the $X_\alpha$ form a partition of $\kappa$, this is a well-defined assignment. Now consider $\beta < \gamma$. We need to show that $B_\beta \cap B_\gamma$ is finite. 
\begin{itemize}
\item If $\alpha(\beta) = \alpha = \alpha(\gamma)$ then $B_\beta \cap B_\gamma \subset S^{\alpha}_{\xi(\beta)} \cap S^{\alpha}_{\xi(\gamma)}$ which is finite, since both sets are elements of the same AD-family $\script{S}_\alpha$.
\item Otherwise, if say $\alpha(\beta) < \alpha(\gamma)$, then $\gamma \notin Y_{\alpha(\beta)}$, so $B_\beta \cap B_\gamma  \subset S^{\alpha(\beta)}_{\xi(\beta)} \cap F_{\gamma}$ is finite since  $\xi(\beta) > \eta \geq \kappa(\gamma,\alpha(\beta))$. \qedhere
\end{itemize}
\end{proof}

\section{The situation under Martin's Axiom}
\label{directproof}

In this section we prove that under MA+$\neg$CH, any binary tree with tops serves as a one-element obstruction set for the class of $(\aleph_0,\aleph_1)$-graphs. For background on Martin's Axiom, see \cite[III.3]{Kunen}. We begin with a sequence of lemmas.

\begin{mylem}
\label{main}
Under MA+$\neg$CH, every $(\aleph_0,\aleph_1)$-graph contains a spanning subgraph isomorphic to a binary tree with tops.
\end{mylem}

\begin{proof}
Let $(A,B)$ be an $(\aleph_0,\aleph_1)$-graph. We want to find an infinite set $T \subset A$ plus a tree order $\prec$ on $T$ such that $\script{T}=(T,\prec)$ is isomorphic to $2^{<\w}$, and an injective map $h \colon B \to \mathcal{B}(\script{T})$ (assigning to each element $b \in B$ a unique branch of $\script{T}$) such that $N(b) \cap h(b)$ is infinite for all $b \in B$. Once we have achieved this, we delete for every $b\in B$ all edges from $b$ to $A\setminus h(b)$ to obtain a binary tree with tops with bipartition $(T,B)$. The remaining vertices in $A \setminus T$ can be easily interweaved with $\script{T}$ as isolated vertices to obtain a spanning such subgraph.

To build this tree $\script{T}$, we consider finite approximations $(T_p,\prec_p)$ to $\script{T}$ (which will be finite initial segments of $\script{T}$), and then use Martin's Axiom to find a consistent way to build the desired full binary tree. Formally, consider the partial order $(\mathbb{P},\leq)$ consisting of tuples $p=(T_p,\prec_p,B_p,h_p)$ such that
\begin{itemize}
\item $T_p \subset A$ finite, and $\prec_p$ a tree-order on $T_p$ such that $(T_p,\prec_p)$ is a  binary tree of some finite height,
\item $B_p\subset B$ finite, and
\item $h_p \colon B_p \to \mathcal{B}((T_p,\prec_p))$ an injective assignment of branches,
\end{itemize}
and $p \leq q$ if
\begin{itemize}
\item $(T_q,\prec_q)$ is an initial subtree of $(T_p,\prec_p)$, 
\item $B_q \subseteq B_p$, and 
\item $h_p$ extends $h_q$ in the sense $h_p(b) \supseteq h_q(b)$ for all $b \in B_q$. 
\end{itemize}

To see that $(\mathbb{P},\leq)$ is ccc
, consider an uncountable collection 
\[\set{p_\alpha=(T_\alpha,\prec_\alpha,B_\alpha,h_\alpha)}:{\alpha < \w_1} \subset \mathbb{P}.\]
By the $\Delta$-System Lemma \cite[III.2.6]{Kunen}, there is a finite root $R \subset B$ and an uncountable $K \subset \w_1$ such that $B_\alpha \cap B_\beta = R$ for all $\alpha \neq \beta \in K$. And since there are only countably many finite subsets of $A$, each with only finitely many possible tree-orders and branch-assignments for $R$, there is an uncountable $K' \subset K$ such that $(T_\alpha,\prec_\alpha) = (T_\beta,\prec_\beta)$ and $h_\alpha \restriction R = h_\beta \restriction R$ for all $\alpha \neq \beta \in K'$. But then for any $\alpha \neq \beta \in K'$, $q=(T_\alpha,\prec_\alpha,B_\alpha \cup B_\beta, h_\alpha \cup h_\beta)$ is a condition below $p_\alpha$ and $p_\beta$ (where we possibly have to increase $T_\alpha$ by one level so a suitable extension of $h_\alpha \cup h_\beta$ can be injective).

Next we claim that for all $b \in B$ and $n\in \w$, the set 
$$D_{b,n}=\set{p \in \mathbb{P}}:{b \in B_p \text{ and } \cardinality{h_p(b) \cap N(b)} \geq n}$$ is dense. To see this, consider any condition $q \in \mathbb{P}$ and suppose $(T_q,\prec_q)$ has height $k$. Choose any subset of $F_b \subset N(b) \setminus T_q$ of size $n$, and extend $T_q$ to a full binary tree $T_p$ of height $k+n$, making sure that $F_b \subset h_p(b)$.

Finally, by Martin's Axiom there is a filter $\mathcal{G}$ meeting each of our $\aleph_1<\cont$ many dense sets in $\mathcal{D}=\set{D_{b,n}}:{b \in B, n \in \w}$. Then
\[\script{T} = (T, \prec) = \p{\bigcup_{p \in \mathcal{G}} T_p, \bigcup_{p \in \mathcal{G}} \prec_p}\] is a countable binary tree, and 
\[h \colon B \to \mathcal{B}(\script{T}), \, b \mapsto \bigcup_{p \in \mathcal{G}} h_p(b)\] is an injective function witnessing that $N(b) \cap h(b)$ is infinite, for our dense sets make sure it has cardinality at least $n$ for all $n\in\N$.
\end{proof}

We remark that it has been shown in either of \cite[Thm.~6]{stepransMA}, \cite[2.3]{OCA} or \cite[4.4]{Luzin} (in historical order) that under MA+$\neg$CH, every almost disjoint family of size $<\continuum$ contains a hidden tree family, which together with our Theorem~\ref{ADsubgraph} and the observations in Section~\ref{sec2} implies the result of Lemma~\ref{main}. 

However, we will now strengthen the claim of Lemma~\ref{main} to hold for \emph{full binary trees with tops}. Clearly, binary trees with tops have fewer edges, and are therefore easier to find as subgraphs than full binary trees with tops. But under Martin's Axiom, it turns out that the additional leeway is not needed. Note though that in the previous theorem, we could find a \emph{spanning} binary tree with tops. In the next theorem, we can obtain full binary trees with tops as subgraphs, but can no longer guarantee that they are spanning.

\begin{mylem}
\label{main2}
Under MA+$\neg$CH, every $(\aleph_0,\aleph_1)$-graph contains a full binary tree with tops as a subgraph.
\end{mylem}

\begin{proof}
Let $(A,B)$ be an $(\aleph_0,\aleph_1)$-graph. We want to find an infinite set $T \subset A$ plus a tree order $\prec$ on $T$ such that $\script{T}=(T,\prec)$ is isomorphic to $2^{<\w}$, and an uncountable $B_T \subset B$ plus an injective map $h \colon B_T \to \mathcal{B}(\script{T})$ (assigning to each element $b \in B_T$ a unique branch of $\script{T}$) such that $h(b) \subseteq N(b)$ for all $b \in B_T$. Once we have achieved this, we delete for every $b\in B_T$ all edges from $b$ to $T\setminus h(b)$ to obtain the desired full binary tree $(T,B_T)$ with tops.

To find this tree $\script{T}$, we build countably many such trees in parallel, which together take care of all $b \in B$. Consider the partial order $(\mathbb{P},\leq)$ consisting of tuples $p=(T_p,\prec_p,B_p,h_p)$ such that
\begin{itemize}
\item $T_p \subset A$ finite, and $\prec_p$ a tree-order on $T_p$ such that $(T_p,\prec_p)$ is a  binary tree of some finite height,
\item $B_p\subset B$ finite,
\item $h_p \colon B_p \to \mathcal{B}((T_p,\prec_p))$ an injective assignment of branches, and
\item $h_p(b) \subset N(b)$ for all $b \in B_p$ 
\end{itemize}
and $p \leq q$ if
\begin{itemize}
\item $(T_q,\prec_q)$ is an initial subtree of $(T_p,\prec_p)$, 
\item $B_q \subseteq B_p$, and 
\item $h_p$ extends $h_q$ in the sense $h_p(b) \supseteq h_q(b)$ for all $b \in B_q$. 
\end{itemize}
As in the proof of Lemma~\ref{main}, this partial order is ccc, and hence so is the finite support product 
\[\prod^{\text{fin}}_{n < \w} \mathbb{P} := \set{\vec{p} \in \mathbb{P}^\w}:{\cardinality{\set{n}:{\vec{p}_n \neq \mathbbm{1}}}<\infty}\]
by \cite[III.3.43]{Kunen}.

We claim that for all $b \in B$, the set $D_b=\set{\vec{p}}:{\exists n \in \w \text{ s.t. } b \in B_{\vec{p}_n} }$ is dense in $\prod^{\text{fin}}_{n < \w} \mathbb{P}$. And indeed, to any condition $\vec{p}$ which does not yet mention $b$ we can simply add $b$ to a free coordinate, even using the empty tree.

So by Martin's Axiom, there is a filter $\mathcal{G}$ meeting every one of our $\aleph_1 < \cont$ many dense sets in $\mathcal{D}=\set{D_b}:{b \in B}$. It follows that \[\set{(T_n,B_n)=\p{\bigcup_{\vec{p} \in \mathcal{G}} T_{\vec{p}_n}, \bigcup_{\vec{p} \in \mathcal{G}} B_{\vec{p}_n}}}:{n \in \N}\] is a countable collection of binary trees with tops, such that $B = \bigcup_{n \in \N} B_n$. Thus, at least one of them, say $B_n$, is uncountable. It follows that in $(T,B_T)=(T_n,B_n)$ we have found our full binary tree with tops embedded as a subgraph as desired. 
\end{proof}

We now proceed to showing that under MA, any two binary trees with tops embed into each other. 
Consider the binary tree $T=2^{<\w}$. A subset $B \subset \script{B}(T)$ of branches is called dense (or $\aleph_1$-dense) if for every $t \in T$ the set $B(t) = \set{b \in B}:{t \in b}$ has size at least $\aleph_0$ (or $\aleph_1$ respectively).

It is well known that the Cantor set $2^\w$ is countable dense homogeneous, i.e.\ for every two countable dense subsets $A,B \subset 2^\w$ there is a self-homeomorphism $f$ of $2^\w$ such that $f(A) = B$. It is also known that under MA+$\neg$CH, this assertion can be strengthened to $\aleph_1$-dense subsets of $2^\w$, see for example \cite[3.2]{aleph1dense} and \cite{steprans}. In the following, we shall see that a mild refinement of this approach, namely adding condition \ref{itemSEP} to the partial order below, also works for $(\aleph_0,\aleph_1)$-graphs of binary type. In this condition (d) below, a level $\script{T}(\alpha)$ of a tree $\script{T}$ is said to \emph{separate} a collection of branches $B \subset \script{B}(\script{T})$ if $B(t) = \set{b \in B}:{t \in b}$ has size at most one for all $t \in \script{T}(\alpha)$.

\begin{mylem}
\label{lem_aleph1dense}
Under MA+$\neg$CH, any two full $\aleph_1$-dense binary trees with tops are isomorphic.  
\end{mylem}

\begin{proof}
Suppose $G=(T_A, A)$ and $H=(T_B,B)$ are two full $\aleph_1$-dense binary trees with tops. For convenience, we treat $a \in A$ as branch of the tree $T_A$. Recall that $a \restriction n$ denotes the unique node of the branch $a$ of height $n$.

It is clear that $A$ and $B$ can be partitioned into $\aleph_1$ many disjoint countable dense sets $\set{A_\alpha}:{\alpha < \w_1}$ and $\set{B_\alpha}:{\alpha < \w_1}$ respectively. Consider the partial order $(\mathbb{P},\leq)$ consisting of tuples $p=(f_p,g_p)$ such that
\begin{enumerate}[label=(\alph*)]
\item $f_p$ is a finite injection with $\operatorname{dom}(f_p) \subset A$ and $\operatorname{ran}(f_p) \subset B$,
\item\label{itemCCC} if $x \in A_\alpha$ then $f_p(x) \in B_\alpha$,
\item $g_p$ is an order isomorphism between $T_A({\leq}n_p)$ and $T_B({\leq}n_p)$ for some $n_p \in \N$,
\item\label{itemSEP} $T_A(n_p)$ separates $\operatorname{dom}(f_p)$ and $T_B(n_p)$ separates $\operatorname{ran}(f_p)$,
\item\label{itemNHOODS} for all $a \in \operatorname{dom}(f_p)$ we have $g_p (a \restriction n_p) = f_p(a) \restriction n_p$,
\end{enumerate}
and define $p \leq q$ if
\begin{itemize}
\item $f_p \supseteq f_q$, and
\item $g_p \supseteq g_q$.
\end{itemize}

To see that $(\mathbb{P},\leq)$ is ccc, consider an uncountable collection 
$$\set{p_\alpha=(f_\alpha,g_\alpha)}:{\alpha < \w_1} \subset \mathbb{P}.$$
Applying the $\Delta$-System Lemma to all sets of the form $I_\alpha = \set{\gamma}:{A_\gamma \cap \operatorname{dom}(f_\alpha) \neq \emptyset}$ (for $\alpha < \w_1$), we obtain a finite root $R$ and an uncountable $K \subset \w_1$ such that $I_\alpha \cap I_\beta = R$ for all $\alpha \neq \beta \in K$.

Since there are only countably many different finite subsets of $A'=\bigcup_{\alpha \in R} A_\alpha$, we may assume that $\operatorname{dom}(f_\alpha) \cap A' = S =  \operatorname{dom}(f_\beta) \cap A'$ for all $\alpha \neq \beta \in K$. And since \ref{itemCCC} implies that there are only countably many choices for $f_\alpha \restriction S$, we may assume that $f_\alpha \restriction S = f_\beta \restriction S$ for all $\alpha \neq \beta \in K$. Finally, since there are only countably many different $g_\alpha$, we may assume that all $g_\alpha \colon T_A(\leq n) \to T_B(\leq n)$ agree. 

But now any two conditions in $\set{p_\alpha}:{\alpha \in K}$ are compatible. By \ref{itemCCC} and the definition of $R$, the map $f=f_\alpha \cup f_\beta$ is a well-defined injective partial map. Extend $g_\alpha$ to an order isomorphism $g \colon T_A(\leq m) \to T_B(\leq m)$ for some sufficiently large $m \geq n$, making sure that $(d)$ and $(e)$ are satisfied. Then $(f,g)$ is a condition below $f_\alpha$ and $f_\beta$, so $(\mathbb{P},\leq)$ is ccc.   

As our dense sets, we will consider
\begin{enumerate}
\item\label{itemTree} $D_n = \set{p \in \mathbb{P}}:{T_A(\leq n) \subset \operatorname{dom}(g_p)}$, for $n\in \N$,
\item\label{itemDOMAIN}  $D_{a}=\set{p \in \mathbb{P}}:{a \in \operatorname{dom}(f_p)}$ for $a \in A$, and
\item\label{itemRANGE}  $D_{b}=\set{p \in \mathbb{P}}:{b \in \operatorname{ran}(f_p)}$ for $b \in B$.
\end{enumerate}

To see that sets in $(1)$ are dense, consider any condition $q = (f_q,g_q) \in \mathbb{P}$ and assume that $\operatorname{dom}(g_q) = T_A(\leq m)$ for some $m < n$. Since for every $t \in T(m)$ there is at most one $a \in \operatorname{dom}(f_q)$ such that $t \in a$ by \ref{itemSEP}, it is clear that we can extend $g_q$ to a function $g_p$ defined on $T_A(\leq n)$ by mapping the upset $t^\uparrow$ in $T_A(\leq n)$ to the corresponding upset of $g_q(t)^\uparrow$ of $T_B(\leq n)$ such that the branch $a \restriction t^\uparrow$ is mapped to $f_q(a) \restriction g_q(t)^\uparrow$. For $f_p = f_q$ we have $p=(f_p, g_p)$ is a condition in $D_n$ below $q$.

To see that sets in $(2)$ are dense, consider any condition $q \in \mathbb{P}$ and assume that $a \notin \operatorname{dom}(f_q)$. Say $\operatorname{dom}(g_q) = T_A(\leq n)$ for a given $n \in \N$. By $(1)$ we may assume that $T_A(n)$ separates $\operatorname{dom}(f_q) \cup \singleton{a}$. Find $t \in T_A(n)$ such that $t \in a$. Note that $a \in A_\alpha$ for some $\alpha < \w_1$. By density of $B_\alpha$, we may pick $ b \in B_\alpha$ extending $g_q(t)$. Then $f_p = f_q \cup \pair{a,b}$ and $g_p = g_q$ gives a condition in $D_a$ below $q$. The argument for $(3)$ is similar.

Finally, Martin's Axiom gives us a filter $\script{G}$ meeting all specified dense sets. But then (\ref{itemDOMAIN}) and (\ref{itemRANGE}) force that $f = \bigcup_{p \in \script{G}} f_p \colon A \to B$ is a bijection, and (\ref{itemTree}) forces that $g=\bigcup_{p \in \script{G}} g_p \colon T_A \to T_B$ is an isomorphism of trees. In combination with property \ref{itemNHOODS}, we have $g[a] = f(a)$ for all $a \in A$, and this means, since $G$ and $H$ were \emph{full} binary trees with tops, that $f \cup g \colon G \to H$ is an isomorphism of graphs.
\end{proof}

\begin{mythm}
Under MA+$\neg$CH, any binary tree with tops embeds into all other $(\aleph_0,\aleph_1)$-graphs as a subgraph.
\end{mythm}

\begin{proof}
Suppose $G=(T_A, A)$ is a binary tree with tops, and $H$ an arbitrary $(\aleph_0,\aleph_1)$-graph. Our task is to embed $G$ into $H$ as a subgraph. By Lemma~\ref{main2}, we may assume that $H=(T_B,B)$ is a full binary tree with tops. 

Our plan is (a) to extend $G$ to a full $\aleph_1$-dense binary tree with tops $G'$, and (b) to find in $H$ a full $\aleph_1$-dense binary tree with tops $H'$ as a subgraph. Then Lemma~\ref{lem_aleph1dense} implies that
$$G \hookrightarrow G' \cong H' \hookrightarrow H,$$
establishing the theorem.

Only item (b) requires proof. For this, we observe that every uncountable set of branches $X$ of a binary tree $T$ contains at least one \emph{complete accumulation point}, i.e.\ a branch $x \in X$ such that for every $t \in x$, the set $B(t) = \set{y \in X}:{t \in y}$ is uncountable. Indeed, otherwise for every $x \in X$ there is $t_x$ such that $B(t_x)$ is countable, and hence $X \subset \bigcup_{t_x \in T} B(t_x)$ is countable, a contradiction. 

It follows that in fact all but at most countably many points of $X$ are complete accumulation points, so without loss of generality, we may assume that every point of $B$ is a complete accumulation point. Consider $T'_B = \bigcup_{b \in B} b \subset T_B$. Then $T'_B$ is a (subdivided) binary tree, so after deleting all non-splitting nodes from $T'_B$, we obtain a full $\aleph_1$-dense binary tree with tops $H'$ as desired. The proof is complete. 
\end{proof}

\section{\texorpdfstring{A third type of $(\aleph_0,\aleph_1)$-graph}%
                        {A third type of A0A1-graph}}

In this section we present a counterexample to the main open question from \cite[\S 8]{NST}, which is our Question~\ref{athirdclass} from the beginning. 

\begin{mythm}
\label{thm_cool}
	Under CH, there is an almost disjoint $(\aleph_0,\aleph_1)$-graph which contains no $(\aleph_0,\aleph_1)$-minor that is indivisible or of binary type.
\end{mythm}

Our proof is inspired by the proof strategy of the following result due to Roitman \& Soukup: \emph{Under CH plus the existence of a Suslin tree, there is an uncountable anti-Luzin AD-family containing no uncountable hidden weak tree families \cite[4.6]{Luzin}.} Note though, that not containing a binary $(\aleph_0,\aleph_1)$-graph as a  minor or just as a subgraph are stronger assertions than not containing an uncountable hidden weak tree family. 

We shall make use of the following lemma.

\begin{mylem}
\label{incomparabletopsAr} Whenever $\script{T}^*$ is Aronszajn, and $B$ an uncountable set of branches of $\script{T}^*$ such that no two elements of $B$ have the same order type, there are incompatible elements $s,t \in \script{T}^*$ both contained in uncountably many branches of $B$.
\end{mylem}

\begin{proof}
The proof follows \cite[4.7]{Luzin}. Consider an Aronszajn tree $\script{T}^*$, and let $B$ be an uncountable set of branches of $\script{T}^*$ such that no two elements of $B$ have the same order type. 

Suppose for a contradiction that whenever $s$ and $t$ are incompatible, then either $B(s) = \set{b \in B}:{ s\in b}$ is countable or $B(t) = \set{b \in B}:{t\in b}$ is countable. Then $S = \set{s}:{B(s) \text{ is uncountable}}$ forms a chain, hence is countable. So there is $\alpha < \w_1$ with $\script{T}^*(\alpha) \cap S=\emptyset$. But now all but countably many elements of B are contained in the countable set $\bigcup_{s \in \script{T}^*(\alpha)} B(s),$ a contradiction. 
\end{proof}

\begin{proof}[Proof of Theorem~\ref{thm_cool}]
Consider an Aronszajn tree $\script{T}^*$, and let $B$ be an uncountable set of branches of $\script{T}^*$ such that no two elements of $B$ have the same order type. 

Using CH, let $\set{\script{T}_\alpha = \p{T_\alpha,<_\alpha}}:{\alpha < \w_1}$ enumerate all trees of countable height whose underlying set is an infinite family of non-empty disjoint subsets of $\N$. For a subset $C \subset \N$ we define $C( \script{T}_\alpha) = \set{t \in T_\alpha}:{C \cap t \neq \emptyset}.$

Let us construct, by recursion on $\alpha < \w_1$, 
\begin{itemize}
	\item families $\set{C_t}:{t\in \script{T}^*(\alpha)}$ of infinite subsets of $\N$, and
    \item countable families $B_\alpha$ of branches of $\script{T}_\alpha$,
\end{itemize}
such that
\begin{enumerate}[label=(\alph*)]
        	\item\label{itemAD} for all $s,t \in T^*$ we have $C_t \subset^* C_s$ if $s<t$, and $C_s \cap C_t =^* \emptyset$ if $s$ and $t$ are incomparable,
    		\item\label{itemDIV} for all $s\neq t \in \script{T}^*(\alpha)$, we have $C_s( \script{T}_\alpha) \cap C_t( \script{T}_\alpha) =^* \emptyset$, and
	\item\label{itemBIN} for all $t \in \script{T}^*(\alpha)$, if $C_t( \script{T}_\alpha)$ contains an infinite chain in $\script{T}_\alpha$, then there is $b \in B_\alpha$ such that $C_t( \script{T}_\alpha) \subset^* b$.
\end{enumerate}

For the construction, suppose for some $\alpha < w_1$ that we have already constructed infinite sets $C_t \subset \N$ for all $t \in \script{T}^*$ of height strictly less than $\alpha$. By \ref{itemAD}, we may pick for every $t \in \script{T}^*(\alpha)$ an infinite pseudo-intersection $D_t$ of the family $\set{C_s}:{s<t}$. Using that every level $\script{T}^*(\alpha)$ of our Aronszajn tree $\script{T}^*$ is countable, find an almost disjoint refinement $\set{D'_t}:{t \in \script{T}^*(\alpha)}$ of $\set{D_t}:{t \in \script{T}^*(\alpha)}$. This can be done either by hand, or by invoking Theorem~\ref{refinement}. Similarly, we can find a further refinement $\set{D''_t}:{t \in \script{T}^*(\alpha)}$ such that $D''_s( \script{T}_\alpha) \cap D''_t( \script{T}_\alpha) =^* \emptyset$ for all $s \neq t \in \script{T}^*(\alpha)$. This takes care of property \ref{itemDIV}.

For \ref{itemBIN}, we use the Aronszajn property to enumerate $\script{T}^*(\alpha) = \set{t_n}:{n \in \N}$. For $n \in \N$, if $D''_{t_n}( \script{T}_\alpha)$ has infinite intersection with some branch of $\script{T}_\alpha$, we pick one such branch $b_n$ and pick an infinite subset $C_{t_n} \subset D''_{t_n}$ such that $C_{t_n}( \script{T}_\alpha) \subset b_n$. Otherwise, we simply put $C_{t_n} = D''_{t_n}$ (and let $b_n$ be an arbitrary branch). This final refinement preserves \ref{itemAD} and \ref{itemDIV}, and after putting $B_\alpha = \set{b_n}:{n \in \N}$, we see that also \ref{itemBIN} is satisfied.

Having completed the construction, we may pick, by \ref{itemAD}, for every branch $b \in B$ an infinite pseudo-intersection $N(b)$ along the branch $b$, i.e.\ $N(b) \subseteq^* C_t$ for all $t \in b$. It follows from \ref{itemAD} that $\set{N(b)}:{b \in B}$ is an almost disjoint family of size $\w_1$. 

Let $G$ be the almost disjoint $(\aleph_0,\aleph_1)$-graph with bipartition $(\N,B)$ where the neighbourhood of $b\in B$ is $N(b)$.

\begin{myclaim*}
Property~\ref{itemBIN} implies that no $(\aleph_0,\aleph_1)$-minor of $G$ is of binary type.
\end{myclaim*}

To see the claim, suppose that $H=(\script{T},X)$ is an $(\aleph_0,\aleph_1)$-minor of $G$ of binary type. Since any non-trivial branch set of the bipartite graph $G$ must contain a vertex from $\N$, we may assume, without loss of generality, that $X \subset B$, and that every branch set $X_t \subset V(G)$ corresponding to a vertex of $t \in \script{T}$ intersects $\N$. Further, there is an injective function $h \colon X \to Br(\script{T})$ mapping points in $X$ to branches of $\script{T}$ such that $N_G(x)(\script{T}) \cap h(x)$ is infinite for all $x \in X$.

However, the tree $\script{T}=\script{T}_\alpha$ appears in our enumeration. Without loss of generality, $X \subset \set{b \in B}:{\operatorname{ht}(b) > \alpha}$. But then \ref{itemBIN} implies that $\operatorname{ran}(h) \subset B_\alpha$, which is countable, contradicting that $X$ is uncountable and $h$ injective. 

\begin{myclaim*}
Property~\ref{itemDIV} implies that every $(\aleph_0,\aleph_1)$-minor of $G$ is divisible.
\end{myclaim*}

Suppose that $H=(A,X)$ is an $(\aleph_0,\aleph_1)$-minor of $G$. As before, we may assume that $X \subset B$ and that the branch sets $X_a \subset V(G)$ for $a \in A$ intersect $\N$. Note that $\script{X}=\set{X_a \cap \N}:{a \in A}$ is the underlying set of uncountably many of our trees $T_\alpha$.

Now by Lemma \ref{incomparabletopsAr}, there are incomparable $s,t \in \script{T}^*$ each contained in uncountably many branches of $X$. Find $\alpha \geq \operatorname{ht}(s),\operatorname{ht}(t)$ such that $\script{X} = T_\alpha$, and find $s',t' \in \script{T}^*(\alpha)$ extending $s$ and $t$ respectively such that $\script{C}=\set{b \in X}:{s' \in b}$ and $\script{D}=\set{b \in X}:{t' \in b}$ are both uncountable.

But then \ref{itemDIV} implies that $C_{s'}(\script{T}_\alpha)$ and its complement witness that $H$ is divisible. Indeed, each $b \in \script{C}$ has co-finitely many of its neighbours in $C_{s'}(\script{T}_\alpha)$, since $N(b) \subset^* C_{s'}$ for all $b \in \script{C}$, and similarly, each $b \in \script{D}$ has co-finitely many of its neighbours in $C_{t'}(\script{T}_\alpha)$, as $N(b) \subset^* C_{t'}$ for all $b \in \script{D}$.
\end{proof}

Since every AD-family built in the above way satisfying \ref{itemAD} is anti-Luzin \cite[4.10]{Luzin}, we obtain the following corollary.

\begin{mycor}
Under CH, there is an uncountable anti-Luzin AD-family which contains no uncountable hidden weak tree families.
\end{mycor}

This improves the corresponding result from \cite[4.6]{Luzin}, where it was proved under the additional assumption of the existence of a Suslin tree. 

\section{More on indivisible graphs}

In this final section, we investigate indivisible graphs in more detail. Our aim is to construct a counterexample to Question~\ref{Uindivisible} from the introduction. First however, we consider the question of when precisely indivisible graphs exist.

We recall two cardinal invariants in infinite combinatorics. The ultrafilter number $\mathfrak{u}$ is the least cardinal of a collection $\script{U}$ of infinite subsets of $\N$ that form a base of some non-principal ultrafilter on $\N$. In formulas,
\[\mathfrak{u} = \min \set{\cardinality{\script{U}}}:{\script{U} \subseteq [\N]^\w \text{ is a base for a non-principal ultrafilter on $\N$}}.\]
(Recall that $\script{U}$ is a base for an ultrafilter $\script{V}$ if $\script{U}\subseteq\script{V}$ and for all $V \in \script{V}$ there is $U \in \script{U}$ such that $U \subseteq V$.) We call $\script{R}\subset[\N]^\w$ a \emph{reaping family} if for all $A\in[\N]^\w$ there is $R\in\script{R}$ such that either $\cardinality{A \cap R}$ or $\cardinality{R \setminus A}$ is finite. 
The \emph{reaping number} $\mathfrak r$ is the least size of a reaping family.
In formulas,
\[\mathfrak r = \min \set{|\script{R}|}:{\script{R}\subseteq[\N]^\omega \text{ and } \forall A\in[\N]^\omega \exists R\in \script{R} ( A\cap R =^* \emptyset \lor R\setminus A =^* \emptyset}.\]

\begin{mythm}
The equality $\mathfrak{u}=\w_1$ implies that indivisible $(\aleph_0,\aleph_1)$-graphs exist, whereas $\mathfrak{r}>\w_1$ implies they do not exist.
\end{mythm}

\begin{proof}
Let $\script{V}$ be a non-principal ultrafilter and let $\set{U_\alpha}:{\alpha < \w_1}$ be a base for $\script{V}$. We will build an indivisible $(\aleph_0,\aleph_1)$-graph with bipartition $(\N,B)$ as follows. Let $B = \set{b_\alpha}:{\alpha<\w_1}$. For every $b_\alpha$ we let $N(b_\alpha)$ be an infinite pseudo-intersection of the family $(U_\beta)_{\beta < \alpha}$. It is easy to check that this yields a graph as desired.

Conversely, if $(\N,B)$ is indivisible, then for every $A \subset \N$, all but countably many elements of $\{N(b):b\in B$\} are almost contained in $A$ or almost disjoint from $A$. It follows that $\{N(b):b\in B\}$ is a reaping family and therefore $\mathfrak{r}=\w_1$.
\end{proof}

In particular, it is well-known (see \cite{vaughn}) that we have
$\w_1 \leq \mathfrak{r} = \pi\mathfrak{u} \leq \mathfrak{u} \leq \cont$, 
where $\pi\mathfrak{u}$ is the least cardinal of a local $\pi$-base of some non-principal ultrafilter on $\N$. Since it is consistent that  $w_1=\mathfrak{u}<\cont$, it follows that CH is independent of the existence of indivisible $(\aleph_0,\aleph_1)$-graphs. However, we do not know whether indivisible graphs exist in the Bell-Kunen model where $\w_1=\pi\mathfrak{u}< \mathfrak{u}$, \cite{bellkunen}.

Lastly, we observe the following connection between indivisible graphs and $\pi$-bases: The neighbourhoods $N(b_\alpha)$ of an $\script{U}$-indivisible $(\aleph_0,\aleph_1)$-graph form a $\pi$-base for $\script{U}$. And conversely, if a family $\set{N_\alpha}:{\alpha < \w_1}$  of infinite subsets of $\N$ forms a $\pi$-base for a \emph{unique} ultrafilter $\script{U}$, then the corresponding $(\aleph_0,\aleph_1)$-graph is indivisible.

We are now ready to answer Question~\ref{Uindivisible} in the negative.

\begin{mythm}
\label{nathansthm}
Assume $\textnormal{CH}$. Let $\script{U}$ be a non-principal ultrafilter on the natural numbers. For every $\script{U}$-indivisible $(\aleph_0,\aleph_1)$-graph $G$ there exists an $\script{U}$-indivisible $(\aleph_0,\aleph_1)$-graph $H$ such that $G \not\preceq H$.
\end{mythm}

\begin{proof}
Using CH, let $\set{U_\alpha}:{\alpha < \w_1}$ be an enumeration of the elements of $\script{U}$, and let $\set{\script{X}_\alpha}:{\alpha < \w_1}$ be an enumeration of all infinite sequences of non-empty disjoint subsets of $\N$. For $\alpha < \w_1$ write $\script{X}_\alpha = \Sequence{X^\alpha_n}:{n \in \N} \in \script{P}\p{\N}^\N$.

Suppose $G$ is a $\script{U}$-indivisible $(\aleph_0,\aleph_1)$-graph with bipartition $(\N,B)$. We write $B=\set{b_\alpha}:{\alpha < \w_1}$. Our graph $H$ will be an $(\aleph_0,\aleph_1)$-graph with bipartition $(\N,C)$ where $C = \set{c_\alpha}:{\alpha < \w_1}$. Our task is to define suitable neighbourhoods $N(c_\alpha)$ for all $\alpha < \w_1$. We will do this as follows. At step $\alpha < \w_1$, choose a neighbourhood $N(c_\alpha) \subset \N$ such that
\begin{enumerate}
\item $N(c_\alpha) \subseteq^*  U_\beta$ for all $\beta \leq \alpha$, and
\item for any $\gamma,\delta \leq \alpha$ there is $n \in N(b_\gamma)$ such that $N(c_\alpha)  \cap X^\delta_n = \emptyset$.
\end{enumerate}
To build the neighbourhood $N(c_\alpha) = \set{m_k}:{k \in \N} $ recursively, enumerate the set $\set{U_\beta}:{\beta \leq \alpha}$ as $\set{U^n}:{n \in \N}$ and $\set{(\beta,\gamma)}:{\beta,\gamma \leq \alpha}$ as $\set{(\beta_n,\gamma_n)}:{n \in \N}$.

To choose $m_k$, note that since the collection $\set{X^{\gamma_k}_n}:{n \in N(b_{\beta_k})}$ is infinite and disjoint, there is an index $n_k \in N(b_{\beta_k})$ such that $X^{\gamma_k}_{n_k} \notin \script{U}$ and $X^{\gamma_k}_{n_k} \cap \set{m_l}:{l < k} = \emptyset$. 
Now pick 
\[m_k \in \bigcap_{l \leq k} \p{U^l \setminus X^{\gamma_l}_{n_l}} \in \script{U}.\]
This choice of $N(c_\alpha) = \set{m_k}:{k \in \N} $ clearly satisfies $(1)$. To see that it satisfies $(2)$, note that $X^{\gamma_k}_{n_k} \cap \set{m_l}:{l < k}=\emptyset$ by our choice of $n_k$, and $X^{\gamma_k}_{n_k} \cap \set{m_l}:{l \geq k}=\emptyset$ by our choice of the $m_l$ for $l \geq k$. This completes the recursive construction of the graph $H$.

\begin{samepage}
\begin{myclaim*} $H$ is $\script{U}$-indivisible.
\end{myclaim*}
This is immediate from $(1)$.
\end{samepage}

\begin{myclaim*} $G$ is not a minor of $H$.
\end{myclaim*}
 Suppose for contradiction that it is. Without loss of generality, we may assume that every vertex on the $\N$-side of $G$ has uncountable degree. Write $V_n,W_\alpha \subset V(H)$ ($n \in \N$, $\alpha < \w_1$) for the branch sets of the vertices in $\N$ and $B$ respectively. By our assumption on the degrees of the vertices on the $\N$-side of $G$, it follows that $V_n \cap \N \neq \emptyset$ for all $n \in \N$. Thus, $\Sequence{V_n \cap \N}:{n \in \N} = \script{X}_\gamma$ for some $\gamma < \w_1$.

Also, since only countably many branch sets can intersect $\N$, there is some $\delta < \w_1$ such that $W_\alpha = \singleton{c_{\beta(\alpha)}}$ for all $\alpha > \delta$. Also, since branch sets must be disjoint, the function $\beta \colon \alpha \mapsto \beta(\alpha)$ is injective.

Let $\eta = \max \Set{\gamma,\delta}$. We claim that for all $\alpha > \eta$, we have $\beta (\alpha) < \alpha$. Indeed, $W_\alpha$ needs to have an edge to all $V_n$ for $n \in N(b_\alpha)$, which requires that $c_{\beta(\alpha)}$ has an edge to  $X^\gamma_n$ for all $n \in N(b_\alpha)$. However, if $\alpha \leq \beta(\alpha)$, then this is impossible, as $(2)$ implies that $N(c_{\beta(\alpha)}) \cap X^\gamma_{n_0} = \emptyset$ for at least one $n_0 \in N(b_\alpha)$.

Thus, we have $\beta (\alpha) < \alpha$ for all $\alpha > \eta$. By Fodor's Lemma \cite[III.6.14]{Kunen}, however, this implies that the map $\beta$ is constant on an uncountable subset of $\w_1$, contradicting its injectivity.
%
\end{proof}

\begin{myquest}
Assume $\textnormal{CH}$. Is it true that for every $\script{U}$-indivisible $(\aleph_0,\aleph_1)$-graph $G$ there exists a $\script{U}$-indivisible $(\aleph_0,\aleph_1)$-graph $H$ such that both $G \not\preceq H$ and $H \not\preceq G$?
\end{myquest}

\subsection*{Acknowledgements}
We would like to thank Lajos Soukup for helpful conversations regarding the proof of Theorem~\ref{thm_cool}, and Alan Dow and Michael Hru\v{s}\'ak regarding the proof of Theorem~\ref{intromain}. 

The second and third author acknowledge support from DAAD project 57156702, ``Universal profinite graphs".

\end{document}